\documentclass[reqno]{amsart}

\newtheorem{theorem}{Theorem}[section]
\newtheorem{lemma}[theorem]{Lemma}
\theoremstyle{definition}
\newtheorem{definition}[theorem]{Definition}

\theoremstyle{remark}

\numberwithin{equation}{section}

\usepackage{amsmath, amsthm}
\usepackage{stackengine}
\usepackage{xcolor}
\usepackage{appendix}
 \usepackage{hyperref}

\providecommand{\keywords}[1]
{
  \small	
  \textbf{\textit{Keywords:}} #1
}

\begin{document}

\title[OPTIMAL CONTROL OF THE 2D CONSTRAINED NAVIER-STOKES EQUATIONS]
{OPTIMAL CONTROL OF THE 2D CONSTRAINED NAVIER-STOKES EQUATIONS}

\author{SANGRAM SATPATHI}
\email{sangramsatpathi0921@gmail.com}

\address{School of Mathematics, Indian Institute of Science Education and Research Thiruvananthapuram, Maruthamala, Thiruvananthapuram, Kerala, 695551, India.}


\begin{abstract}
   We study the 2D Navier–Stokes equations within the framework of a constraint that ensures energy conservation throughout the solution. By employing the Galerkin approximation method, we demonstrate the existence and uniqueness of a global solution for the constrained Navier–Stokes equation on the torus $\mathbb{T}^2$. Moreover, we investigate the linearized system associated with the 2D-constrained Navier-Stokes equations, exploring its existence and uniqueness. Subsequently, we establish the Lipschitz continuity and Fréchet differentiability properties of the solution mapping. Finally, employing the formal Lagrange method, we prove the first-order necessary optimality conditions. 
\end{abstract}


\maketitle

\section{Introduction}
Incompressible Navier-Stokes equations are used to understand the dynamics of an incompressible viscous fluid. These equations were proposed by C. Navier in 1822 and were later derived by G. Stokes. By solving these equations, we can predict how the fluid's speed changes over time and in different places, based on the initial and boundary states. These equations have many practical uses, from studying aerodynamics to modeling blood flow in the body but the basic mathematical
question of the existence of a unique global-in-time solution to these parabolic PDEs on a bounded
domain in $\mathbb{R}^3$
still remains open due to the non-linear convective term. The existence of a unique global-in-time solution to the Navier-Stokes equations on $\mathbb{R}^2$ has
been known for a long time. Ladyzhenskaya \cite{Ladyzhenskaia1959SolutionT} proved an inequality to control the non-linear term in a bounded domain in $\mathbb{R}^2$ which was later used to prove the existence and uniqueness of the solution to Navier-Stokes equations. The study of 2D-constrained Navier-Stokes equations adds another factor to consider, such as a restriction on the energy of the solution known as $L^2$-energy. The reason why we study this constrained problem is that these equations are expected to provide a better approximation to the incompressible Euler equations. This is because, for the Euler equations, the energy of solutions (which are smooth enough) remains constant. The study conducted in \cite{article} considered two-dimensional Navier-Stokes equations as in the Caglioti et al. \cite{art},associated with the same energy constraint
as in Caffarelli et al. \cite{Caffarelli2008NonlocalHF} and Rybka \cite{rybka_2006}. To be specific, they considered the
Navier-Stokes equations projected on the tangent space of the manifold M, where
\begin{align*}
    \mathrm{M}=\{ u\in H(\mathbb{T}^2) : |u|_H^2 =1\}.
\end{align*}
Here $H$ is the space of square-integrable, divergence-free, mean zero vector fields
on a torus $\mathbb{T}^2$.They examined the following form 
\begin{align*}
    \frac{d u(t)}{dt}+[\nu A u(t)+B(u(t))]=0.
\end{align*}
The authors have shown that if the initial data belongs to the space  ${V} \cap \mathrm{M}$ then the solution of the above equation $u(t) $ stays on the manifold $\mathrm{M}$ for all time $t$.
In this paper, we consider the Navier-Stokes equations of the form 
\begin{align*}
\begin{cases}
    \frac{\partial{u(x,t)}}{\partial{t}}-\nu\Delta u(x,t)+(u(x,t)\cdot\nabla)u(x,t)+\nabla p(x,t)=f(u(x,t))\\
\nabla.u(x,t)=0.\\
u(x,0)=u_0(x),
\end{cases}
\end{align*}
subject to the same constraint as in \cite{rybka_2006}\cite{Caffarelli2008NonlocalHF}\cite{article}\cite{artic}. we prove the existence of the solution only on a torus by the Galerkin approximation method. Our proof does not hold in $\mathbb{R}^2$.
\par We are interested in the problem 
$$
\frac{d u(t)}{dt}+[A u(t)+B(u(t))]={f}(u(t)), \quad t \geq 0,\;\;u(0)=u_0.
$$
where $u\in{H}$. Similar to the approach in \cite{article}, we project the aforementioned equation onto the tangent space of $\mathrm{M}$, resulting in the following.
\begin{align*}{\label{eq:11}}
    \begin{cases}
        \frac{du}{dt}+[{A}u +B(u)] =|\nabla u|^{2} u +{f},\\
        u(0)=u_0.
    \end{cases}
\end{align*}
In \cite{phdthesis}, the author focuses on investigating optimal control problems related to the non-stationary Navier-Stokes equations. He introduced a study on solution mapping and presented some valuable results of it for the unsteady Navier-Stokes equations. In this paper, we will prove those results for the 2D-constrained Navier-Stokes equations. We added a control term to the right-hand side of the above equation.
We linearized the system and investigated the existence and uniqueness of its solution. We also analyze several significant properties of the solution mapping. These results will have a crucial role in studying the control of 2D-constrained Navier-Stokes equations. We employ the formal Lagrange method \cite{tröltzsch2010optimal} to establish the first-order necessary optimality conditions. The optimization problem is defined as follows:
 $$\min J(y,U)$$
 subject to the state equation 
 \begin{align*}
      y_{t}+A y+B(y)-|\nabla y|^{2} y={U}\\
      y(0)=y_0
 \end{align*}
 $U\in U_{ad}$.\\
 Where \[
J(y,U):=\frac{1}{2} \int_{0}^{T}|A^{1 / 2} y(t)|_{H}^{2} dt+\frac{1}{2} \int_{0}^{T}|U(t)|_{V}^{2} dt
\]
and 
 $$U_{ad}:=\{ U \in T_u \mathrm{M} : |U|_{V} \text{ is bounded}\}.$$
 In this context, $U$ represents the control variable and $y$ represents the solution of the state equation. In section (6), we introduce the Lagrange functional and examine its directional derivative in relation to both the control and state. Ultimately, we conclude the section by demonstrating the necessary optimality condition.
\section{Constrained Navier-Stokes equation}
\subsection{General notations}
Let $\Omega$ be a bounded domain in $\mathbb{R}^2$, $\mathbb{R}^2$, or $\mathbb{T}^2$. For $b \in [1,\infty]$ and $k \in \mathbb{N}$, we denote the Sobolev space and Lebesgue spaces of $\mathbb{R}^2$ by $W^{k,p}(\Omega,\mathbb{R}^2)$ (or $W^{k,p}$) and $L^p(\Omega,\mathbb{R}^2)$ (or $L^p$), respectively. Additionally, we define ${H}^2$ as $W^{k,2}$.
Let $\mathbb{T}^2$ represent the bounded periodic domain, which can be visualized as a two-dimensional torus. Now, we will introduce the following spaces:
\begin{align*}
    \mathcal{L}_0^2 &= \{u\in L^2(\mathbb{T}^2,\mathbb{R}^2) : \int_{\mathbb{T}^2} u(x)\, dx = 0\}, \\
    {H} &= \{u \in \mathcal{L}_0^2 : \nabla \cdot u = 0\}, \\
    {V}  &= {H}^1 \cap {H}.
\end{align*}
The scalar product and norm of ${H}$ can be represented as the $L^2$ scalar product and $L^2$ norm, respectively, denoted by:
\begin{align*}
    \langle u, v \rangle_{{H}} \quad \text{or} \quad \langle u, v \rangle \quad \text{and} \quad |u|_{{H}} \quad \text{or} \quad |u|.
\end{align*}
Moreover, the scalar product and norm of ${V} $ are also referred to as the ${H}^1$ scalar product and norm, respectively.\\Let us defined the Stokes operator and discuss some important things about it. We represent the Stokes operator as ${A}: D({A}) \rightarrow {H}$, where ${A}$ maps from the domain $D({A})$ to the Hilbert space ${H}$. The Stokes operator is defined as follows:
\begin{align*}
{A}u := -\Delta u,
\end{align*}
 The domain $D({A})$ of the Stokes operator is defined as the intersection of the Hilbert space ${H}$ and the Sobolev space ${{H}^2(\mathbb{T}^2)}$, denoted as:
\begin{align*}
D({A}) = {H} \cap {{H}^2(\mathbb{T}^2)}=E.
\end{align*}
Since $ \langle {A}u\;,\;u\rangle=(|\nabla u|)^2\;\;\text{for}\; u\in D({A})$, so the Stokes operator is non-negative operator. The stokes operator is also a self-adjoint operator.

\subsection{Operators and their properties}
From now onwards we identify our domain as a two-dimensional torus $\mathbb{T}^2$.
We can introduce a continuous trilinear map $b: L^p \times W^{1,q} \times L^r \rightarrow \mathbb{R}$ defined as follows:

\begin{align*}
b(u,v,w) = \sum_{i,j=1}^2 \int_{\Omega} u^i \frac{\partial v^j}{\partial x^i} w^j  dx,
\end{align*}
\\where $ \;p,q,r \in [1,\infty]$
such that $\frac{1}{p}+\frac{1}{q}+\frac{1}{r}\leq1$.
\par Let $ B:{V} \times {V} \rightarrow \mathrm{V'}$ be the bilinear map such that,
 \begin{align*}
     \langle B(u,v)\;,\;\phi\rangle=b(u,v,\phi),\text{for} \;u,v,\phi\in V.
 \end{align*}
 When considering $u\in{{V}}$, $v\in E$, and $w \in {{H}}$, we can establish the following inequality:

\begin{align*}
    |b(u,v,\phi)| \le \sqrt{2}|u|_{{H}}^{\frac{1}{2}}|u|_{{V}}^{\frac{1}{2}}|v|_{{V}}^{\frac{1}{2}}|v|_{E}^{\frac{1}{2}}|w|_{{H}}.
\end{align*}

Hence we can uniquely extend the trilinear map $b$ to operate on the triple\\ ${V}\times E\times{H}$.

Furthermore, the map $B$ can be extended uniquely to a bounded operator denoted as:

\begin{align*}
    B: {V}\times E \rightarrow {H}.
\end{align*}

The properties of the tri-linear map and bilinear map are the following:
\begin{align*}
    b(u,u,u)&=0 ,\;u\in {V}.\\
   b(u,w,w)&=0,\;u\in {V},w\in H^1.\\
   \langle B(u,u)\;,\;{A}u\rangle_{H}&=0,\;u\in D(A).
\end{align*}
The proof of the above results can be found in \cite{temam1977navier}.
\begin{lemma}
    Let $\mathcal{Q}: {V}  \rightarrow {H}$ be defined by
    $$\mathcal{Q} (u):=|\nabla u|^{2} u, \quad u \in {V} .
$$
Then there exists $C>0$ such that for $u_{1}, u_{2} \in {V} $,
$$
\left|\mathcal{Q} \left(u_{1}\right)-\mathcal{Q} \left(u_{2}\right)\right|_{{H}} \leq C\left|u_{1}-u_{2}\right|_{{V} }\left(\left|u_{1}\right|_{{V} }+\left|u_{2}\right|_{{V} }\right)^{2}
$$
\end{lemma}
\begin{proof}
    $$
\begin{aligned}
\left|\mathcal{Q} \left(u_{1}\right)-\mathcal{Q} \left(u_{2}\right)\right|_{{H}} &= \left|\left|\nabla u_{1}\right|^{2} u_{1}-\left|\nabla u_{2}\right|^{2} u_{2}\right|_{{H}}\\
& =\left.|| \nabla u_{1}\right|^{2} u_{1}-\left|\nabla u_{1}\right|^{2} u_{2}+\left|\nabla u_{1}\right|^{2} u_{2}-\left.\left|\nabla u_{2}\right|^{2} u_{2}\right|_{{H}} \\
& =\left|\left|\nabla u_{1}\right|^{2}\left(u_{1}-u_{2}\right)+\left(\left|\nabla u_{1}\right|^{2}-\left|\nabla u_{2}\right|^{2}\right) u_{2}\right|_{{H}} \\
& \leq\left|\nabla u_{1}\right|^{2}\left|u_{1}-u_{2}\right|_{{H}}+\left(\left|\nabla u_{1}\right|+\left|\nabla u_{2}\right|\right)\left|\left|\nabla u_{1}\right|-\left|\nabla u_{2}\right|\right|\left|u_{2}\right|_{{H}}\\
& \leq C\left[\left|\nabla u_{1}\right|^{2}\left|u_{1}-u_{2}\right|_ {V} +\left(\left|\nabla u_{1}\right|+\left|\nabla u_{2}\right|\right)\left|\nabla\left(u_{1}-u_{2}\right)\right|\left|u_{2}\right|_ {V} \right] \\
& \leq C\left|u_{1}-u_{2}\right|_ {V} \left[\left|u_{1}\right|_{{V} }^{2}+\left|u_{2}\right|_{{V} }^{2}+\left|u_{1}\right|_ {V} \left|u_{2}\right|_ {V} \right] \\
& \leq C\left|u_{1}-u_{2}\right|_ {V} \left(\left|u_{1}\right|_ {V} +\left|u_{2}\right|_ {V} \right)^{2} .
\end{aligned}
$$
Here we have used the fact that ${V} $ is continuously embedded in ${H}$.
\end{proof}
\subsection{The deterministic model}
The 2D Navier-Stokes equations are given as follows:
\begin{align}
\begin{cases}
    \frac{\partial{u(x,t)}}{\partial{t}}-\nu\Delta u(x,t)+(u(x,t).\nabla)u(x,t)+\nabla p(x,t)=f(u(x,t)).\\
\nabla.u(x,t)=0.\\
u(x,0)=u_0(x).
\end{cases}
\end{align}
Here, we consider the domain $\mathcal{O}$ and time interval $[0, T]$ for all $T > 0$. The variables $x \in \mathcal{O}$ and $t \in [0, T]$ represent spatial coordinates and time, respectively.

In this context, $u: \mathcal{O} \rightarrow \mathbb{R}^2$ denotes the velocity field, while $p: \mathcal{O} \rightarrow \mathbb{R}$ represents the pressure field of the fluid.
By employing the conventional approach of applying the projection map to the aforementioned problem, we attain the following form,
$$
\frac{d u(t)}{dt}+[A u(t)+B(u(t))]={f}(u(t)), \quad t \geq 0,\;\;u(0)=u_0.
$$
Let us represent the set of divergence-free $\mathbb{R}^2$-valued functions with unit $L^2$ norm as follows:
\begin{align*}
    \mathrm{M}=\{u\in{H}:|u|_{L^2}=1\}.
\end{align*}
The tangent space of it is defined as:

\begin{align*}
    T_u\mathrm{M}  = \{v\in{H} : \langle v, u \rangle_{H} = 0\}, u\in \mathrm{M}.
\end{align*}
We define an orthogonal projection map $\pi_u:{H}\rightarrow T_u\mathrm{M}$ by,
\begin{align*}
    \pi_u(v)=v-\langle v\;,\;u\rangle_{H}.
\end{align*}
Several assumptions will be made about the function $f$, it is globally Lipschitz, has a linear growth, belongs to the tangent space of the manifold $\mathrm{M} $ and $f(u(t))\in L^2 (0, T; V), t\in [0,T]$.
\par Let $F(u) = {A}u + B(u,u)-f(u)$ be a function, and $\mathcal{F}(u)$ be the projection of $F(u)$ onto the tangent space $T_u\mathrm{M} $. Then,

$$
\begin{aligned}
\mathcal{F} (u) & =\pi_{u}(F(u)) \\
& =F(u)-\langle F(u), u\rangle_{H} u \\
& =A u+B(u)-f(u)-\langle A u+B(u)-{f}(u), u\rangle_{{H}} u \\
& =A u-|\nabla u|_{H}^{2} u+B(u)-{f}(u) .
\end{aligned}
$$
Hence, by projecting the equation onto the tangent space $T_{u} \mathrm{M} $, we derive the following constrained Navier-Stokes equations.
\begin{align}{\label{eq:1}}
    \begin{cases}
        \frac{du}{dt}+[{A}u +B(u)] =|\nabla u|^{2} u +{f},\\
        u(0)=u_0.
    \end{cases}
\end{align}
\section{Existence and uniqueness}
The proof of the existence of  the solution of ({\ref{eq:1}})  is based on the Galerkin approximation method.
\par Let $\left\{e_{i}\right\}_{i=0}^{\infty}$ be the orthonormal basis in ${H}$ composed of eigen vectors of ${A}$ corresponding to the eigen values $\left\{\lambda_{i}\right\}_{i=0}^{\infty}$ . Where ${A}$ is a positive self-adjoint operator.
\begin{align*}
    {A}e_{i}=\lambda_i e_{i}.
\end{align*}
Let ${H}_n$ be the subspace of ${H}$ equipped with the norm inherited from ${H}$.
$$
{H}_{n}:=\operatorname{Linspan}\left\{e_{1}, \ldots, e_{n}\right\}.
$$
$P_n$ be the projection operator on ${H}$ defined by
\begin{align*}
P_{n} u = \sum_{i=1}^n\langle u, e_i\rangle_{{H}} e_{i},\quad u \in {{H}}.
\end{align*}
Utilizing the notations established above, we can examine the Galerkin approximation of the constrained Navier-Stokes equations in the ${H}_{n}$ space:  
\begin{align}{\label{eq:2}}
\begin{cases}
& \frac{d u_{n}}{d t}=-\left[P_{n} {A} u_{n}+P_{n} B (u_{n})\right]+\left|\nabla u_{n}\right|^{2} u_{n}+P_n f\left(u_{n}\right). \\
& u_{n}(0)= P_{n} u_{0}.
\end{cases}
\end{align}
First, we will show that the solution will stay inside the sphere $\mathrm{M}$, that is $|u_n|_{{H}}^2\le 1$.
\begin{lemma}
    Let $u_0\in{V} \cap\mathrm{M}$,  then $|u_n|_{{H}}^2\le 1$, where $u_n$ is the solution of ({\ref{eq:2}}).
\end{lemma}
\begin{proof}
    \begin{align*}
 &\frac{1}{2}\frac{d}{dt}|u_n(t)|_{{H}}^2= \Big\langle -P_n{A}u_n(t) - P_n B(u_n(t)) + |\nabla u_n|^2u_n + P_n f(u_n)\; ,\; u_n\Big\rangle_H\\
 &\Rightarrow\frac{1}{2} d\left|u_{n}(t)\right|_{{H}}^{2}=-\left|u_{n}(t)\right|_{{V} }^{2} d t+\left|\nabla u_{n}(t)\right|^{2}\left|u_{n}(t)\right|_{{H}}^{2} d t
 \\
 &\Rightarrow d\left[\left|u_{n}(t)\right|_{{H}}^{2}-1\right]=2\left|u_{n}(t)\right|_{{V} }^{2}\left[\left|u_{n}(t)\right|_{{H}}^{2}-1\right] d t .
    \end{align*}
    Integrating both sides from $0$ to $t$, we get,
    $$
     \left|u_{n}(t)\right|_{{H}}^{2}-1=\left[\left|u_{n}(0)\right|_{{H}}^{2}-1\right] \exp \left[2 \int_{0}^{t}\left|u_{n}(s)\right|_{{V} }^{2} d s\right]
    $$
    Since $|u_n(0)|_{{H}}=|P_nu_0|_{{H}}\le |u_0|_{{H}}=1\text{ and } \int_{0}^{t}\left|u_{n}(s)\right|_{{V} }^{2} d s<\infty$, we get
    $$
    |u_n(t)|_{{H}}^2\le 1\;\;\;\forall \;t\;<\infty$$.
\end{proof}
\subsection{Passage to the limit} We will obtain \textit{a priori} estimates independent of $n$ for the functions $u_n$ and then pass the limit.
\par By taking the inner product of Equation (\ref{eq:2}) with ${A}u_n$, we obtain the following expression,
\begin{align}{\label{eq:3}}
\begin{cases}
\left\langle\frac{d u_{n}}{d t}, {A} u_{n}\right\rangle_{H}= -\left\langle A u_{n}, A u_{n}\right\rangle _H-\left\langle P_n B (u_{n}), A u_{n}\right\rangle_{H} +& \langle |\nabla u_n|^2 u_n, A u_n\rangle_H \\&+ \langle P_n f(u_n),u_n \rangle_H.
\end{cases}
\end{align}
Because the Stokes operator and the projection operator $P_n$ are self-adjoint,the function $f(u_n)\in L^2(0,T; V)$ and using $\langle B(u_n),{A}u_n\rangle_{{H}}=0$ ,we have the following,
\begin{align*}
    \frac{1}{2} \frac{d}{dt}\left|u_{n}\right|_{V}^{2}
    &= -\left\langle A u_{n} - \left|\nabla u_{n}\right|^{2} u_{n}, A u_{n} - \left|\nabla u_{n}\right|^{2} u_{n}\right\rangle - \left\langle A u_{n} - \left|\nabla u_{n}\right|^{2} u_{n}, \left|\nabla u_{n}\right|^{2} u_{n}\right\rangle \\
    &\quad + \left\langle f\left(u_{n}\right), u_{n}\right\rangle_{V} \\
    &= -\left|A u_{n} - \left|\nabla u_{n}\right|^{2} u_{n}\right|^{2} - \left\langle A u_{n} - \left|\nabla u_{n}\right|^{2} u_{n}, \left|\nabla u_{n}\right|^{2} u_{n}\right\rangle + \left\langle f\left(u_{n}\right), u_{n}\right\rangle_{V}. 
\end{align*}
$\text{Since } \left| A u_{n} - \left| \nabla u_{n} \right|^{2} u_{n} \right|_{H}^{2} \geq 0$, we can neglect this term in the previous equation, allowing us to express it as follows,
\begin{align*}
        \frac{1}{2} \frac{d}{dt}\left|u_{n}\right|_{V}^{2}
    \leq - \langle Au_n-\left|\nabla u_{n}\right|^{2} u_{n}, \left|\nabla u_{n}\right|^{2} u_{n}\rangle + \left\langle f\left(u_{n}\right), u_{n}\right\rangle_{V}.
\end{align*}
Now consider the term, 
\begin{align*}
    &\left\langle A u_{n}-\left|\nabla u_{n}\right|^{2} u_{n},\left|\nabla u_{n}\right|^{2} u_{n}\right\rangle \\
 &=\left\langle A u_{n},\left|\nabla u_{n}\right|^{2} u_{n}\right\rangle-\left\langle\left|\nabla u_{n}\right|^{2} u_{n},\left|\nabla u_{n}\right|^{2} u_{n}\right\rangle\\
 &=|\nabla u_n|^2 \langle Au_n , u_n\rangle -|\nabla u_n|^4|u_n|^2\\
 &\le |\nabla u_n|^4-|\nabla u_n|^4=0. 
\end{align*}
Since $|u_n|^2\leq 1$, the above calculation is valid. Hence using this estimation we have,
\begin{align*}
    \frac{1}{2} \frac{d}{d t}\left|u_{n}\right|_{V}^{2} \leq\left\langle f\left(u_{n}\right), u_{n}\right\rangle_{V}.
\end{align*}
Taking the integration from $0$ to $t,0<t\le T,$ we have,
\begin{align*}
    \left|u_{n}(t)\right|_{V}^{2}-\left|u_{n}(0)\right|_{V}^{2} \leq  2\int_{0}^{t}\left\langle f\left(u_{n}(s)\right), u_{n}(s)\right\rangle_{V} ds.
\end{align*}
Using Young's Inequality we obtain for a given $\varepsilon$,
\begin{align*}
     \int_{0}^{t}\left\langle f\left(u_{n}\right), u_{n}\right\rangle_{V} &\leq \varepsilon\left|f\left(u_{n}\right)\right|_{L^{2}(0, t ; V)}^{2}+\frac{1}{4 \varepsilon}\left|u_{n}\right|_{L^{2}(0, t ; V)}^{2} \\ &\leq C_{1}+C_{2} \int_{0}^{t}\left|u_{n}\right|_{V}^{2} d s .
\end{align*}
Since $f$ has linear growth. Hence 
\begin{align*}
\left|u_{n}(t)\right|_{V}^{2} \leq C_{1}+C_{2} \int_{0}^{t}\left|u_{n}(s)\right|_{V}^{2} d s \\
\end{align*}
By applying Gronwall's inequality, we can have $u_n\in L^{\infty}(0, T; V)$ for all $n$.\\
Again consider (\ref{eq:3}),
\begin{align*}
\left\langle\frac{d u_{n}}{d t}, {A} u_{n}\right\rangle_{H}= -\left\langle A u_{n}, A u_{n}\right\rangle _H-\left\langle P_n B (u_{n}), A u_{n}\right\rangle_{H} &+ \langle |\nabla u_n|^2 u_n, A u_n\rangle_H \\&+ \langle P_n f(u_n),u_n \rangle_H\\
 \Rightarrow\frac{1}{2} \frac{d}{dt}\left|u_{n}\right|_{V}^{2}=-|A u_{n}|^{2}-|u_n|_{V}^4 +\left\langle f\left(u_{n}\right), u_{n}\right\rangle_{V}\\
 \Rightarrow\frac{1}{2} \frac{d}{dt}\left|u_{n}\right|_{V}^{2}+|A u_{n}|^{2}\leq +\left\langle f\left(u_{n}\right), u_{n}\right\rangle_{V}.
\end{align*}
Taking integration from $0$ to $ T <\infty$ we obtain,
\begin{align*}
 |Au_n|_{L^2(0,T;H)}^2\leq C_{1} +C_{2}\int_0^T |u_n|_{V}^2<\infty.
\end{align*}
The above term is finite because of $u_n\in L^{\infty}(0, T; V)$. So by the above estimation, we have $u_n\in L^2(0, T; D(A))$ for all $n$. Therefore there exists a subsequence of $u_n$, denoted again the same as $u_n$ such that, ${u_n}$ converges to $u_*$ in weak* topology of $L^{\infty}(0,T;V)$ and $u_n$ converges to $u$ weakly in $L^2(0,T;D(A))$.
\par Now we aim to demonstrate the equality of both limits, that is $u=u_*$. Hence by using definitions of weak and weak* convergence, we have,
 \begin{align}\label{eq:10}\forall\;\; v \in L^{1}\left(0, T ; V^{\prime}\right),\quad\int_{0}^{T}\Big\langle u_{n}-u_{*}, v\Big\rangle d t \rightarrow 0 \;\text{as} \;n \rightarrow \infty.\end{align}
 Again 
 \[ \int_{0}^{T}\Big\langle u_{n}-u, v\Big\rangle d t \rightarrow 0 \quad \forall\; v \in L^{2}\left(0, T ; D(A)^{\prime}\right).\]
 Now since $L^{2}\left(0, T ; V^{\prime}\right) \subset L^{1}\left(0, T ; V^{\prime}\right)$ and $\forall\; v \in L^{1}\left(0, T ; V^{\prime}\right) \Rightarrow \forall\; v \in L^{2}\left(0, T ; V^{\prime}\right), $
 therefore from (\ref{eq:10}),
  \noindent\[\int_{0}^{T}\langle u_{n}-u_{*}, v\rangle d t \rightarrow 0\quad\forall\; v \in {L}^{2}\left(0,T;V^{\prime}\right).\]
  Considering the inclusion $D(A) \subset V$, we have $V^{\prime} \subset (D(A))^{\prime}$. Consequently, we can infer that $L^{2}\left(0, T ; V^{\prime}\right) \subset L^{2}\left(0, T ;(D(A))^{\prime}\right)$. Hence, we can conclude that: \begin{align*}
  \begin{cases}
      &\int_{0}^{T}\langle u_{n}-u_{*}, v\rangle \, dt \rightarrow 0, \quad \forall\, v \in L^{2}\left(0, T ;(D(A))^{\prime}\right), \\
      &\int_{0}^{T}\langle u_{n}-u, v\rangle , dt \rightarrow 0,\; \forall\, v \in L^{2}\left(0, T ;(D(A))^{\prime}\right).
  \end{cases}
\end{align*}

Hence we get $u=u_*$.\\
The following results can be found in p-183 of \cite{temam1977navier}.\\
\textbf{A compactness theorem in Banach spaces.} Let $X_{0}, X, X_{1}$, be three Banach spaces such that
$$
X_{0} \subset X \subset X_{1},
$$

where the injections are continuous and

$$
\begin{aligned}
& X_{i} \text { is reflexive, } i=0,1, \\
& \text { the injection } X_{0} \rightarrow X \text { is compact. }
\end{aligned}
$$

Let $T>0$ be a fixed finite number and $\alpha_0$ and $\alpha_1$ are two finite numbers such that $\alpha_i>0$ for $i=0,1$. Consider the space 

$$
\begin{aligned}
& \mathcal{Y}=\mathcal{Y}\left(0, T ; \alpha_{0}, \alpha_{1} ; X_{0}, X_{1}\right) \\
& \mathcal{Y}=\left\{v \in L^{\alpha_{0}}\left(0, T ; X_{0}\right), \quad v^{\prime}=\frac{d v}{d t} \in L^{\alpha_{1}}\left(0, T ; X_{1}\right)\right\}
\end{aligned}
$$
It is obvious that
$$
\mathcal{Y} \subset L^{\alpha_{0}}(0, T ; X)
$$
With a continuous injection.
\begin{theorem}\label{thm:A2} Under the above assumptions the injection of $\mathcal{Y}$ into $L^{\alpha_0}(0, T ; X)$ is compact.
\end{theorem}
\begin{proof} 
See Theorem 2.1 \cite{temam1977navier}.
\end{proof}
We will use the above results to show the strong convergence. Now, considering the definitions:
\begin{align*}
&X_{0}=D(A)=H \cap H^{2}\left(\mathbb{T}^{2}\right),\\
&X =V= H\cap H^{1}\left(\mathbb{T}^2\right),\\
&X_{1} =H,
\end{align*}
we have the inclusion $X_{0} \subset X \subset X_1$, and the compact embedding $X_{0} \hookrightarrow X_{1}$.

Let us define the set:
\begin{align*}
    \mathbf{\mathcal{Y}}=\left\{v \in L^{2}(0, T ; D(A)) \mid v^{\prime} \in L^{2}(0, T ; H)\right\}.
\end{align*}

It follows that $\mathbf{\mathcal{Y}}\hookrightarrow L^2(0,T:V)$ is a compact embedding. Consequently, we can conclude that $u_n\rightarrow u $ strongly in $L^2(0,T;V)$.\\
 Hence we are allowed to pass the limit.
 To pass the limit, consider the following equation:
\begin{align*}
&\frac{d u_{n}}{d t}=-P_{n} A u_{n}-P_{n} B\left(u_{n}\right)+\mid \nabla u_{n}|^{2} u_{n}+P_{n}f\left(u_{n}\right).
\end{align*}
Let us consider a function $\Psi$ that is continuously differentiable and all the derivative is bounded and satisfies $\Psi(T) = 0$. Then,
   \begin{multline*}
       \int_{0}^{T}\Big\langle\frac{d u_{n}}{d t}, \Psi(t) e_{j}\Big\rangle_{H} d t =-\int_{0}^{T}\Big\langle   P_{n} A u_n(t) ,\Psi(t) e_{j}\Big\rangle _{H} d t\\ -
       \int_{0}^{T}\Big\langle  P_{n} B u_{n}(t), \Psi(t) e_{j}\Big\rangle_{H} d t\\
        +\int_{0}^{T}\Big\langle  \left|\nabla u_{n}(t)\right|^{2} u_n(t), \Psi(t) e_j\Big\rangle_{H} d t
+\int_{0}^{T}\Big\langle  P_{n} f\left(u_{n}(t)\right), \Psi(t) e_j\Big\rangle_{H} d t.
    \end{multline*}
To demonstrate the convergence term by term, let us first consider the following term:
$$
    \int_{0}^{T}\Big\langle\frac{d u_{n}}{d t}, \Psi(t) e_{j}\Big\rangle_{H} d t =-\int_{0}^{T}\Big\langle u_{n}(t), \Psi^{\prime}(t) e_{j}\Big\rangle_{H} d t -\Big\langle u_{n}(0), \Psi(0) e_{j}\Big\rangle_{H} .$$\\
    Hence we have,
    \begin{align*}
    -\int_{0}^{T}\Big\langle u_{n}(t), \Psi^{\prime}(t) e_j\Big\rangle_H d t &=\Big\langle  u_{n}(0), \Psi(0) e_{j}\Big\rangle_{H}-\int_{0}^{T}\Big\langle  P_n A u_{n}(t), \Psi(t) e_j\Big\rangle_H \\ &-\int_{0}^{T}\Big\langle  P _{n} B\left(u_{n}(t)\right), \Psi(t) e_j\Big\rangle_H d t
 \\&+\int_{0}^{T}\Big\langle  \left|\nabla u_{n}(t)\right|^{2} u_n(t), \Psi(t) e_j\Big\rangle_{H} d t 
 \\&+\int_{0}^{T}\Big\langle  P_{n} f\left(u_{n}(t)\right), \Psi(t) e_j\Big\rangle_{H} d t.
\end{align*}
To show
$$\int_{0}^{T}\Big\langle  u_{n}(t), \Psi^{\prime}(t) e_ j\Big\rangle_{H} dt \rightarrow -\int_{0}^{T}\Big\langle  u(t), \Psi^{\prime}(t) e_j\Big\rangle_{{H}} ,$$let us consider following:
\begin{align*}
\left|\int_{0}^{T}\Big\langle  u_{n}(t), \Psi^{\prime}(t) e_{j}\Big\rangle_{H}-\int_{0}^{T}\Big\langle  u_{n}(t), \Psi^{\prime}(t) e_{j}\Big\rangle_{H} \right| \leq \int_{0}^{T} \left|\Big\langle  u_{n}(t)-u(t), \Psi^{\prime}(t) e_j\Big\rangle_{H} \right|.
\end{align*}
By utilizing the Cauchy-Schwarz inequality, the aforementioned term can be expressed as follows:
\begin{equation*}
\begin{aligned}
\leq \int_{0}^{T}\big|u_{n}(t)-u(t)\big|_{H}\big| \Psi^{\prime}(t) e_j\big|_{H} \, dt
&\leq C\int_{0}^{T}\big|u_{n}(t)-u(t)\big|_V\big|\Psi^{\prime}(t) e_j\big|_{H} \, dt \\
&\leq \tilde{C}\big|u_{n}(t)-u(t)\big|_{L^2(0, T; 
 V)} \rightarrow 0 \text{ as } n \rightarrow \infty.
\end{aligned}
\end{equation*}
Again consider the term,
\begin{align*}
&\int_{0}^{T}\langle P_{n} B(u_{n}(t)), \Psi(t) e_j\rangle_{H} dt - \int_{0}^{T}\langle B(u(t)), \Psi(t) e_j\rangle_H dt \\
&\leq \int_{0}^{T}\left|\langle P_{n} B(u_{n}(t)) - B(u(t)), \Psi(t) e_j\rangle_{H}\right| dt \\
&\leq C \left(\int_{0}^{T}\left|P _{n} B(u_{n}(t)) - B(u(t))\right|_{H} dt\right) \\
&\leq C \left[\int_0^T |B(u_n(t)) - B(u(t))|_{H} dt + \int_0^T |P_n - I| |B(u(t))|_{H} dt\right] \rightarrow 0.
\end{align*}
In the above calculation, we utilized the fact that $P_n$ is a contraction and as $n\rightarrow \infty$, $P_n$ converges to the identity map $I$. Now let's consider the term below,
\begin{equation*}
\begin{aligned}
&\left|\int_{0}^{T}\Big\langle\left|\nabla u_{n}(t)\right|_{H}^{2} u_{n}(t), \Psi(t) e_{j}\Big\rangle_{H}-\int_{0}^{T}\Big\langle|\nabla u(t)|_H^{2} u(t) , \Psi(t) e_{j}\Big\rangle_{H } \right|\\&\quad\quad\quad\quad\quad\quad\quad\quad\quad\quad\quad\quad\quad
 \leq \int_{0}^{T}\left|\nabla u_n(t)|_{H}^{2} u_{n}(t)-|\nabla u(t)|_{H}^{2} u(t)\right|_{H}\left|\Psi(t) e_{j}\right|_{H} d t \\
&\quad\quad\quad\quad\quad\quad\quad\quad\quad\quad\quad\quad\quad\quad\quad\leq C \int_{0}^{T}\left|\left|\nabla u_{n}(t)\right|_{H}^{2} u_{n}(t)-\left|\nabla u_{n}(t)\right|_{H}^{2} u(t)\right|_{H} d t \\&
\quad\quad\quad\quad\quad\quad\quad\quad\quad\quad\quad\quad\quad\quad\quad\leq\tilde{C} \int_{0}^{T}\left|u_{n}-u\right|_{V} [\left|u_{n}\right|_{V}+\left| u \right|_{V}]^{2}. \\
 &\text { Since } u_{n} \rightarrow u \text{ in } L^{2}(0, T; V) \text { so }\left|u_{n}\right|_{V},|u|_{V}<\infty.
 \end{aligned}
 \end{equation*}
 Hence, the right-hand side of the above estimation tends toward zero. Now, let's consider the next term:
\begin{align*}
\bigg| \int_{0}^{T}\Big\langle P_{n} f\left(u_{n}\right), \Psi(t) e_j\Big\rangle_{H}  d t&-\int_{0}^{T}\Big\langle f(u) , \Psi(t) e_{j}\Big\rangle_{H} d t \bigg|\\
&\leq \int_{0}^{T}\left|\Big\langle P_{n} f\left(u_{n}\right)-f(u), \Psi(t) e_{j}\Big\rangle_{H}\right| dt\\ 
&\leq \int_{0}^{T}\left|{P_n} f\left(u_{n}\right)-f(u)\right|_{H}\left|\Psi(t) e_{j}\right|_{H} d t\\
&\leq C \int_{0}^{T}\left|f\left(u_{n}\right)-f(u)\right|_{H} d t+C \int_{0}^{T}\left|P_{n} f(u)-f(u)\right|_{H} d t \\
&\leq \tilde{C} \int_{0}^{T}\left|u_{n}-u\right|_{V}^{2} d t+\tilde{C}\int_{0}^{T}\left|P_{n}-I\right||f(u)|_H dt.
\end{align*}
Based on the previous arguments, we can show the right-hand side goes to zero of the above inequality. However, we still need to show that the $Au_n$ term converges.
$$
\begin{aligned}
\int_{0}^{T}\Big\langle A u_{n}-A u, \Psi(t)e_j\Big\rangle_{H} d t= & \int_{0}^{T}\Big\langle\left(u-u_{n}\Big\rangle, \Psi(t)e_j\right)_{H} d t \\
& \leq \int_{0}^{T}\Big\langle\nabla\left(u_{n}-u\right), \nabla \Psi(t)e_j\Big\rangle_{H} d t \\
& \leq C\int_0^T  | u_{n}-u|_{H}| \nabla \Psi(t)e_j |_{H} d t \\
& \leq C |u_n-u|_{L^2(0,T;V)}
\end{aligned}
$$
Since $u_n\rightarrow u$ in $L^2(0, T; V)$ hence we have the right-hand side of the above inequalty goes to zero. Therefore we have can pass the limit to the following equation,
$$
\begin{aligned}
-\int_{0}^{T}\Big\langle u(t), \Psi^{\prime}(t) e_{j}\Big\rangle_{H} d t  =&\Big\langle u(0), \Psi(0) e _j\Big\rangle-\int_{0}^{T}\Big\langle A u(t), \Psi(t) e_j\Big\rangle dt \\&-\int_{0}^{T}\Big\langle B u(t), \Psi(t) e_j\Big\rangle_{H} d t +\int_{0}^{T}\Big\langle| \nabla u(t) |_{H}^2 u, \Psi(t) e_j\Big\rangle_{H} d t \\&+\int_{0}^{T}\Big\langle f(u), \Psi(t) e_j\Big\rangle_H d t
\end{aligned}
$$
holds for all $e_j$. So it will hold for all $v=$ finite linear combinations of $e_j$ while passing the limit it is valid for all $v \in H$.

\noindent Finally, we need to show $u$ holds the equation,
$$
\begin{aligned}
& \frac{d u}{d t}=-A u-B(u) +|\nabla u|^{2} u+f(u) .\\
& u(0)=u_{0}.
\end{aligned}
$$
Multiply by $\Psi$ and continue by similar and then comparing we have $u$ satisfies the above equation.\\
Now for the uniqueness part consider the following, 
Let $u_1$ and $u_2$ are the solution of,\\
$$
\begin{cases}
    \frac{d u_1}{dt}=-A u_1-B(u_1)+|\nabla u_1|^{2} u_1+f(u_1). \\
 u_1(0)=u_{10}. 
 \end{cases}
 $$
$$
\begin{cases}
 \frac{d u_2}{dt}=-A u_2-B(u_2)+|\nabla u_2|^{2} u_2+f(u_2) .\\
 u_2(0)=u_{20}.
 \end{cases}
 $$
 $
\\
 \implies\; \frac{d u_{1}}{d t}-\frac{d u_{2}}{d t}=-A\left(u_{1}-u_{2}\right)-B\left(u_{1}\right)+B\left(u_{2}\right)  +\left|\nabla u_{1}\right|^{2} u_{1}-\left|\nabla u_{2}\right|^{2} u_{2} +f\left(u_{1}\right)-f\left(u_{2}\right) .\\
u_1(0)-u_2(0)=u_{10} -u_{20}.
\\
\implies u^{\prime}=-A u-B\left(u_{1}\right)+B\left(u_{2}\right)+\left|\nabla u_{1}\right|^{2} u_{1}-\left|\nabla u_{2}\right|^{2} u_{2} +f(u_1)-f(u_2).\\
 u(0)=u_{10} -u_{20}.\\
\\
\bigg[\text{Taking}\; u=u_1-u_2\bigg.]\\
$
\\
Taking inner product with $u$ in both sides we have,
\begin{align*}
    \left\langle u^{\prime}, u\right\rangle_{H}=&-<A u, u>_{H}-b\left( u, u_{2}, u\right) +<\left|\nabla u_{1}\right|^{2} u_{1}-\left|\nabla u_{2}\right|^{2} u_{2}, u>_{H}\\&+<f\left(u_{1}\right)-f\left(u_{2}\right), u>_H\\
 \Rightarrow \frac{1}{2} \frac{d}{d t}|u|_{H}^{2}&=-|\nabla u|_{H}^{2}-b\left(u, u_{2}, u\right) +\Big\langle\left|\nabla u_{1}\right|^{2} u_{1}-\left|\nabla u_{2}\right|^{2} u_{2}, u\Big\rangle_{H} \\&+\Big\langle f\left(a_{1}\right)-f\left(u_{2}\right), u\Big\rangle_H.
 \end{align*}
 Consider,\\
\begin{equation*}
\begin{aligned}
\Big\langle  \left|\nabla u_{1}\right|^{2} u_{1}-\left|\nabla u_{2}\right|^{2} u_{2} u\Big\rangle_H\leq\left|\left[\left|\nabla u_{1}\right|^{2} u_{1}-\left|\nabla u_{2}\right|^{2} u_{2}\right]\right|_{H}|u|_{H} \\
 \leq C\left|u_{1}-u_{2}\right|_{V}\left[\left|u_{1}\right|_{V}+\left|u_{2}\right|_{V}\right]^{2}|u|_{H} \\
=C|u|_{V}\left[\left|u_{1}\right|_{V}+\left|u_{2}\right|_{V}\right]^{2}|u|_{H}\\
\leq C \varepsilon|u|_{V}^{2}+\frac{C}{4 \varepsilon}|u|_{H}^{2}\left[\left|u_{1}\right|_{V}+\left|u_{2}\right|_{V}\right]^{4}.
\end{aligned}
\end{equation*}
\\Again we have,
 \begin{align*}
     & \left|\Big\langle f\left (u_{1}\right)-f\left(u_{2}\right), u_{1}-u_{2}\Big\rangle_{H}\right| \leq K|u|_{H}^{2}.\quad\text{[Since $f$ is Lipschitz.]}
 \end{align*}
 $\&$
 \begin{align*}
 \left|b\left(u, u_{2}, u\right)\right| &\leq \sqrt{2}|u|_{H}^{1 / 2}|u|_{V}^{1 / 2}\left|u_{2}\right|_{V}^{1 / 2}\left|u_{2}\right|_{E}^{1 / 2}|u|_H\quad\;\quad\quad\\
     &\leq \sqrt{2} C_{1}|u|_{H}|u|_{V}\left|u_{2}\right|_{V}^{1 / 2}
\left|u_{2}\right|_{E}^{1 / 2}\\
&=\sqrt{2} C_{1} \varepsilon|u|_{V}^{2}+\frac{\sqrt{2} C_{1}}{4 \varepsilon}|u|_{H}^{2}\left|u_{2}\right|_{V}
\left|u_{2}\right|_{E}.
\end{align*}
 Writing altogether we have,\\
\begin{align*}
     \frac{1}{2}\frac{d}{dt}|u(t)|_{H}^{2} \leq-|u|_{V}^{2}+\sqrt{2} C_{1} \varepsilon|u|_{V}^{2}+C \varepsilon|u|_{V}^{2}+\frac{C}{4 \varepsilon}|u|_{H}^{2}\left[\left|u_{1}\right|_{V}+\left|u_{2}\right|_{V}\right]^{4}  +K|u|_{H}^{2} \\+\frac{\sqrt{2}C_1}{4 \varepsilon}|u|_{H}^{2}|u_2|_V |u_2|_E.
\end{align*}
 $\text { Take } {C_2}=\max \left\{\sqrt{2} C_{1}, C, K\right\}$
 \begin{equation*}
\begin{aligned}
 \frac{1}{2}\frac{d}{dt}|u(t)|_{H}^{2}\leq -|u|_{V}^{2}+ C_{2} \varepsilon|u|_{V}^{2}+C_2 \varepsilon|u|_{V}^{2}+\frac{C_2}{4 \varepsilon}|u|_{H}^{2}\left[\left|u_{1}\right|_{V}+\left|u_{2}\right|_{V}\right]^{4}  +C_2|u|_{H}^{2} \\+\frac{C_2}{4 \varepsilon}|u|_{H}^{2}|u_2|_V |u_2|_E.
 \end{aligned}
 \end{equation*}
 Choose  $\varepsilon $ such that $(2 C_2 \varepsilon-1)<0$,
\text { So } $\varepsilon<\frac{1}{2C_2}$.
Therefore,
\begin{align*}
    \frac{d}{d t}|u|_{H}^{2} \leq {{C}} |u|_{H}^{2}.
\end{align*}
    $\text{Where}\; {C}=2[\;\;\frac{C_2}{4 \varepsilon}\left[\left|u_{1}\right|_{V}+\left|u_{2}\right|_{V}\right]^{4}  +\frac{C_2}{4 \varepsilon}|u_2|_V |u_2|_E\;+\;C_2\;\;]$.\\
So,
\begin{align*}
&\;\frac{d}{d t}|u(t)|_{H}^{2} \leq C|u(t)|_{H}^{2}\implies\frac{d}{d t}\left\{\exp \left(-\int_{0}^{t} {{C}} d s\right)|u(t)|_{H}^{2}\right\} \leq 0  \Rightarrow \quad|u(t)|_{H}^{2} \leq 0 \\
& \Rightarrow |u(t) |_{H}=0 \Rightarrow u_{1}(t)=u_{2}(t) \quad\forall\; t \in[0, T].
\end{align*}
Hence the solution is unique.\\
\section{Linearized equations}
We will need some of the results about the linearized equations.
Let $u$ be a solution of,
\begin{align*}
    &u_{t}+A u+B(u)-\mid \nabla u|_H^{2} u=U.\\
    & u(0)=u_0.\\
\end{align*}
Let $\bar{u}$ be the solution of  $A \bar{u}+B(\bar{u})-|\nabla \bar{u}|^{2} \bar{u}=0.$\\
$\text{Now}$ let $\omega=u-\bar{u}
\text { or } u=\omega+\bar{u}$.
So putting the value of $u$ in the first equation we have,
$$
\begin{aligned}
&(\omega+\bar{u})_{t}+A(\omega+\bar{u})+B(\bar{u}+\omega)-\left.| \nabla(\bar{u}+\omega)\right|^{2}(\bar{u}+\omega)=U.
\end{aligned}
$$
Now for equilibrium point $\bar{u}_{t}=0$.
So,\\
\begin{align}\label{eq:*}
    \omega_{t}+A \omega+A \bar{u}+B(\bar{u}+\omega)-\mid \nabla(\bar{u}+\omega) \mid ^2(\bar{u}+\omega) =U.
\end{align}

\noindent Here ,\\$B(\bar{u}+\omega)=(\bar{u}+\omega) \cdot \nabla)(\bar{u}+\omega)\\=(\bar{u} \cdot \nabla)(\bar{u}+\omega)+(\bar{\omega} \cdot\nabla)(\bar{u}+\omega)\\ 
=(\bar{u} \cdot \nabla) \bar{u}+{(u \cdot \nabla) \omega+(\omega \cdot \nabla) \bar{u}}{+(\omega+\nabla) \omega}.
$
\\Since we are linearizing so we can ignore the nonlinear term. Hence,

$$
\begin{aligned}
B(\bar{u}+\omega) =B(\bar{u})+(\bar{u} \cdot \nabla) \omega+(\omega \cdot \nabla) \bar{u} =B(\bar{u})+B^{\prime}(\bar{u}) \omega.
\end{aligned}
$$
Now from (\ref{eq:*})  we have , \\
$$
\begin{aligned}
& \begin{aligned}
\omega_{t}+A \omega+A \bar{u}+B(\bar{u})+B^{\prime}(\bar{u}) \omega  -|\nabla \bar{u}|^{2}(\bar{u}+\omega)
-|\nabla \omega|^{2}(\bar{u}+\omega) -2\langle\nabla \bar{u}, \nabla \omega\rangle(\bar{u}+\omega)
 =U.
\end{aligned} \\
& \text { Since } A \bar{u}+B(\bar{u})-|\nabla \bar{u}|^{2} \bar{u}=0,\;\text{and ignoring the nonlinear terms we have,}
\end{aligned}
$$
\begin{align*}
    \omega_{t}+A \omega+B^{\prime}(\bar{u}) \omega  -|\nabla \bar{u}|^{2}\omega -2\langle\nabla \bar{u}, \nabla \omega\rangle\bar{u}
 =U.
\end{align*}
Let us define a map,
\begin{align*}
&\Phi_{T} : {X_{T}} \longrightarrow L^2(0, T ; {H})\; \text{ by\;} \Phi_{ T}(\omega)(x, t)=G(\omega)(x, t).
\end{align*}
Where $X_T=C([0,T],{V} )\cap L^2(0, T; E).$\\
\noindent Then $\Phi_{T}$ is globally lipschitz.
To prove it let us consider, $ \omega_1,\omega_2 \in X_T $ and then
\begin{align*}
\left|\Phi_{T}\left(\omega_{1}\right)-\Phi_{T}\left(\omega_{2}\right)\right|_{L^2(0;T; {H})}&=\left|G\left(\omega_{1}\right)-G\left(\omega_{2}\right)\right|_{L^{2}(0, T ; {H})}\\
&=\left.\left|U-B^{\prime}(\bar{u}) \omega_{1}+\right| \nabla \bar{u}\right|_{{H}}^{2}  \omega_{1}+2\left\langle\nabla \bar{u}, \nabla \omega_{1}\right\rangle \bar{u} 
-U+B^{\prime}(\bar{u}) \omega_{2} 
\\&-|\nabla \bar{u}|_{{H}}^{2} \omega_{2} 
 -2\left\langle\nabla \bar{u}, \nabla \omega_{2}\right\rangle
 \bar{u}|_{L^2(0, T; {H})} \\
&=| B^{\prime}(\bar{u}) \omega_{2}-B^{\prime}(\bar{u}) \omega_{1}
 + 2\left\langle\nabla \bar{u}, \nabla \omega_{1}-\nabla \omega_{2}\right\rangle \bar{u}\\&+|\nabla \bar{u}|^2 \left(\omega_{1}-\omega_{2}\right)|_{L^2(0, T; {H})}\\&\le
\left[\int_{0}^{T}\left| | \nabla \bar{u} |_{{H}}^{2}\left[\omega_{1}-\omega_{2}\right]\right|_{H} ^{2} d t\right]^{1 / 2} 
\\&+\left[\int_{0}^{T}| B^{\prime}(\bar{u}) \omega_{2}-B^{\prime}(\bar{u}) \omega_{1}|_{{H}}^2\right]^{1/2}\ \\&+\left[\int_{0}^{T} \left|2\left\langle\nabla \bar{u}, \nabla \omega_{1}-\nabla \omega_{2}\right\rangle \bar{u}\right|_{{H}}^{2} d t\right]^{1 / 2}.
\end{align*}
Let us denote these 3 terms by $A_1,A_2,A_3$ respectively.\\
So,
\begin{equation*}
\begin{aligned}
A_{1}^{2}&=\left[\int_{0}^{T}\left| | \nabla \bar{u} |_{{H}}^{2}\left[\omega_{1}-\omega_{2}\right]\right|_{{H}} ^{2} d t\right] \le\int_{0}^{T}|\nabla \bar{u}|_{{H}}^{4}\left|\omega_{1}-\omega_{2}\right|_{{H}}^{2} d t \\
&=|\nabla \bar{u}|^{4} \int_{0}^{T}\left|\omega_{1}-\omega_{2}\right|_{{H}}^{2} d t. \le C_{1}|\nabla \bar{u}|_{{H}}^{4}\left|\omega_{1}-\omega_{2}\right|_{X_{T}}^{2}  \\
 A_{1} &\leq C_1|\nabla \bar{u}|_{{H}}^{2}\left|\omega_{1}-\omega_{2}\right|_{X_{T}}.
 \end{aligned}
 \end{equation*}
\noindent  Consider,
\begin{align*}
 A_{2}^{2}&=\int_{0}^{T}\left|B^{\prime}(\bar{u}) \omega_{1}-B^{\prime}(\bar{u}) \omega_{2}\right|_{{H}}^{2} d t \\
&=\int_{0}^{T} \mid(\bar{u} \cdot \nabla) \omega_{1}+\left(\omega_{1} \cdot \nabla\right) \bar{u}-(\bar{u} \cdot \nabla) \omega_{1}
 -\left.\left(\omega_{2} \cdot \nabla\right) \bar{u}\right|_{{H}} ^{2} d t \\
 &=\int_{0}^{T}\left|(\bar{u} \cdot \nabla)\left(\omega_{1}-\omega_{2}\right)+\left(\left(\omega_{1}-\omega_{2}\right) \cdot \nabla\right) \bar{u}\right|_{{H}}^{2} dt \\
 &\leq \int_{0}^{T}\left|(\bar{u}\cdot \nabla)\left(\omega_{1}-\omega_{2}\right)\right|_{{H}}^{2} d t+\int_{0}^{T}\left|\left(\omega_{1}-\omega_{2}\right) \cdot \nabla \bar{u}\right|_{{H}}^{2} d t \\
\implies &A_2\leq C_2 |\bar{u}|_E |\omega_1-\omega_2|_{X_T}.
\end{align*}
\noindent Again,
\begin{align*}
\;\;A_{3}^{2}= & \int_{0}^{T}\left|2\left\langle\nabla \bar{u}, \nabla \omega_{1}-\nabla \omega_{2}\right\rangle_{{H}} \bar{u}\right|_{{H}}^{2} d t \\
\leq & 4 \int_{0}^{T}|\bar{u}|_{{H}}^{2}|\nabla u|_{{H}}^{2}\left|\nabla\left(\omega_{1}-\omega_{2}\right)\right|_{{H}}^{2} d t \\
\le & 4|\bar{u}|_{{H}}^{2}|\nabla \bar{u}|_{{H}}^{2} \int_{0}^{T}\left|\nabla\left(\omega_{1}-\omega_{2}\right)\right|_{{H}}^{2} d t \\
\leq & 4|\bar{u}|_{{H}}^{2}|\nabla \bar{u}|_{{H}}^{2}\left|\omega_{1}-\omega_{2}\right|_{X_{T}}^{2} \\
\implies  A_{3} & \le 2 C_{3}|\bar{u}|_{{H}}|\nabla \bar{u}|_{{H}}\left|\omega_{1}-\omega_{2}\right|_ {X_{T}}.\\
\end{align*}
Hence,
\begin{align*}
    \left|\Phi_{T}\left(\omega_{1}\right)-\Phi_{T}\left(\omega_{2}\right)\right|_{L^2(0;T; {H})}\leq K |\omega_1-\omega_2|_{X_T}.\\
\end{align*}
$\text{Where}\; K=[\;2 C_{3}|\bar{u}|_{{H}}+C_2 |\bar{u}|_E+C_1|\nabla \bar{u}|_{{H}}^{2}\;]\;< \;\infty.$\\
Therefore $\Phi_T$ is Globally Lipschitz. Hence Theorem 1.9.1 of \cite{cartan1983differential} says that the Linearized system has a unique global solution. 
\section{The control-to-state mapping}
Now, we will take one step further towards achieving optimal control of the state equations. Our focus will be on studying \textit{control-to-state mapping}, which involves mapping the right-hand side of the equations to their corresponding solutions.
\begin{definition}{\textbf{(Solution mapping)}}
    Let $U\in L^2(0, T; V)$  denote the control. Consider the system {(\ref{eq:1})}. The mapping from the control variable $U$ to the corresponding weak solution $y$, where $y$ is the solution of equation (\ref{eq:1}) with the control right-hand side and a fixed initial value $y_0$, is denoted by $S$. In other words, we represent this mapping as $y = S(U)$.
\end{definition}
\textbf{Note:} We will use $C$ to represent the constant, and we often use the same symbol to represent other constants.
\subsection{Continuity and Differentiability}
\begin{lemma}
    The {\textit{control-to-state mapping}} is Lipschitz continuous from $L^2(0, T; V)$ to $L^2(0, T; D({A}))\cap L^{\infty}(0, T; V)$.
\end{lemma}
\begin{proof}
 Let $y_1, y_2$ be two solutions of (\ref{eq:2}) with the same initial
value $y_0$ and associated with the control functions $U_1, U_2, y_i = S(U_i)$. Denote by $y$ and $u$ the difference between solutions and control, i.e. $y=y_1-y_2$ and $U=U_1-U_2$.
We subtract the corresponding operator equations and take the inner product with ${A}y$ and we have the following,
$$
\begin{aligned}
\frac{1}{2} \frac{d}{dt}\left|y(t)\right|_{V}^{2} =  &-|A y(t)|^{2} + \langle B(y_{2}(t))-B(y_{1}(t)), A y(t)\rangle  \\ &+\langle|\nabla y_{1}(t)|^{2} y_{1}(t)-|\nabla y_{2}(t)|^{2} y_{2}(t), A y(t)\rangle  + \langle U(t), A y(t)\rangle
\end{aligned}
$$
Consider the following term,
$$
\begin{aligned}
B\left(y_{2}\right)-B\left(y_{1}\right) & =-\left[B\left(y\right)-B\left(y_{2}\right)\right] \\
& =-\left[B\left(y_{1}\right)+B^{\prime}\left(y_{2}\right) y\right]
\end{aligned}
$$
Hence
$$
\begin{aligned}
\left\langle B\left(y_{2}\right)-B\left(y_{1}\right), A y\right\rangle=&-\left\langle B\left(y\right)+B^{\prime}\left(y_{2}\right) y, A y\right\rangle \\
&=\left[0+\left\langle B^{\prime}\left(y_{2}\right) y, A y\right\rangle\right] \\
&=-[\;b(y_2,y,{A}y)+b\left(y, y_{2}, A y\right)\;].
\end{aligned}
$$
and since $\big|\big|\nabla y_{1}\big|^{2} y_{1} - \big|\nabla y_{2}\big|^{2} y_{2} \big| \leq C\big| y\big|_{V}$, so we have

$$
\begin{aligned}
\big|\langle|\nabla y_{1}|^{2} y_{1}-|\nabla y_{2}|y_{2}, Ay\rangle\big| \leq {C}|Ay|^{2} .
\end{aligned}
$$
Again, using the previous results of the trilinear map $b$, for any $u\in{{V}}$, $v\in E$, and $\phi \in {{H}}$, we have the following inequality:

\begin{align*}
    |b(u,v,\phi)| \le \sqrt{2}|u|_{{H}}^{\frac{1}{2}}|u|_{{V}}^{\frac{1}{2}}|v|_{{V}}^{\frac{1}{2}}|v|_{E}^{\frac{1}{2}}|\phi|_{{H}}.
\end{align*}

So, we have,
\begin{align*}
    \left|b\left(y_2, y, A y\right)\right| \leq C |y|_{E}^{2}.
\end{align*}
Similarly,
\begin{align*}
    \left|b\left(y, y_2, A y\right)\right| \leq C |y|_{E}^2.
\end{align*}
Now by Young's inequality for a given $\varepsilon$ we have
\begin{align*}
   \langle U, y\rangle \le C\varepsilon |y|_{E}^2 + \frac{C}{4\varepsilon} |U|_{{V} }^2.
\end{align*}
We will choose $\varepsilon$ in such a way that $-1$ will dominate all other coefficients of $|Ay|^2$, that is,

\begin{align*}
\frac{1}{2} \frac{d}{d t}|y(t)|_{V}^{2}+k|Ay|^{2} \leq C|U|_{{V} }^2.
\end{align*}
Here $k>0$. \\
 Therefore by taking the integration from 0 to $T$, we can say 
 $$|y|_{L^{\infty}(0 , T ; V)}^2 \leq C|U|_{L^2(0 , T ; V)}^2$$.

Again 
\begin{align*}
   k|A y|^{2} \leq C |U|_{V}^{2}.
\end{align*}
So by taking the integration from $0$ to $T$ we have $|y|_{L^2(0, T; D(A))}^2\le C |U|_{L^2(0, T; V)}^2$.
Hence the solution mapping $S(U)=y$ is lipschitz continuous from $L^{2}(0, T ; V)$ to $L^{2}(0, T ; D(A)) \cap L^{\infty}(0, T ; V)$.
\end{proof}
Now, we will demonstrate the Fréchet differentiability of the solution mapping.
\begin{lemma}
    The control-to-state mapping exhibits Fréchet differentiability, acting as a mapping from $L^{2}(0, T ; V)$ to $L^{2}(0, T ; D(A)) \cap L^{\infty}(0, T ; V)$. The derivative at $\bar{U} \in L^{2}(0, T ; V)$ in the direction $h \in L^{2}(0, T ; V)$ is expressed as $S'(\bar{U})h = y$, where $y$ represents the weak solution of
    \begin{align*}
    y_{t}+A y+B^{\prime}(\bar{y} ) y  -|\nabla \bar{y} |^{2}y -2\langle\nabla \bar{y} , \nabla y\rangle\bar{y} 
 &=h.\\
 y(0)&=0.
\end{align*}
with $S(\bar{U} )=\bar{y}$.
\end{lemma}
\begin{proof}
    Define $y=S(\bar{U} +h)$. Hence,
    \begin{align}{\label{eq:5}}
         \bar{y}_{t}+A \bar{y}+B(\bar{y})-|\nabla \bar{y}|^2 \bar{y}=\bar{U} , \end{align}
    \begin{align}{\label{eq:4}}
          y_{t}+A y+B(y)-|\nabla y|^{2} y=\bar{U} +h.
    \end{align}
    Let  $y-\bar{y}=d$, or $y=d+\bar{y}$. Put the value of $d$ in (\ref{eq:4}) we obtain
    $$
\begin{aligned}
& d_{t}+\bar{y}_{t}+A d+A \bar{y}+B(d+\bar{y})-|\nabla(d+\bar{y})|^{2}(d+\bar{y}) =\bar{U} +h.
\end{aligned}
$$
From the term $|\nabla(d+\bar{y})|^{2}(d+\bar{y})$ we have
$$
\begin{aligned}
& |\nabla(d+\bar{y})|^{2}(d+\bar{y}) = \langle\nabla d+\nabla \bar{y}, \nabla d+\nabla \bar{y}\rangle(d+\bar{y}) \\
= & \left|\nabla d|^{2} d+2\langle\nabla d, \nabla \bar{y}\rangle d+2\langle\nabla d, \nabla \bar{y}\rangle \bar{y}\right. +|\nabla \bar{y}|^{2} d+|\nabla \bar{y}|^{2} \bar{y} +|\nabla d|^{2} \bar{y}.
\end{aligned}
$$
Since $B(d+\bar{y})=  B(d)+B^{'}(\bar{y}) d+B(\bar{y})$, the following expression can be written:
$$
\begin{aligned}
 d_{t}+A d+ B^{\prime}(\bar{y}) d-|\nabla d|^{2} d+B(d) +\bar{y}_{t}+A \bar{y}&-|\nabla \bar{y}|^{2} \bar{y} +B(\bar{y})=h+2\langle\nabla d, \nabla \bar{y}\rangle d\\& +|\nabla \bar{y}|^{2} d+2\langle\nabla d, \nabla \bar{y}) \bar{y} +|\nabla d|^{2} \bar{y}+\bar{U} 
\end{aligned}
$$
Since $S(\bar{U} )=\bar{y}$, then we have,
\begin{align*}
     d_{t}+A d+B^{\prime}(\bar{y}) d-|\nabla \bar{y}|^{2} d -2\langle\nabla \bar{y}, \nabla d\rangle \bar{y}=h-B(d)+|\nabla d|^{2} d +& 2\langle\nabla \bar{y}, \nabla d\rangle \bar{d}\\&+\mid \nabla d \mid^{2} \bar{y}.
\end{align*}
We split $d$ into $d=z+r$, where $z$ and $r$ are the weak solutions of the following systems respectively
\begin{align*}
    \begin{cases}
        z_{t}+A z+B^{\prime}(\bar{y}) z-|\nabla \bar{y}|^{2} z-2\langle\nabla \bar{y}, \nabla z\rangle \bar{y}=h ,\\
        z(0)=0.\\
    \end{cases}
\end{align*}
\begin{align*}
    \begin{cases}
                 r_{t}+A r+B^{\prime}(\bar{y}) r-|\nabla \bar{y}|^{2} r -2\langle\nabla \bar{y} , \nabla r\rangle \bar{y}= -B(d)+|\nabla d|^{2} d 
 +2\langle\nabla d, \nabla \bar{y}\rangle d+|\nabla d|^{2} \bar{y}\\
 r(0)=0.
    \end{cases}
\end{align*}
Let $X=L^{2}(0, T ; D(A)) \cap L^{\infty}(0, T ; V)$. To finalize the proof, it is sufficient to show the following: 
\begin{align}{\label{eq:7}}
    \frac{\left|y-\bar{y}-z\right|_X}{\left|h\right|_{L^2(0, T; V)}} \rightarrow 0 \quad \text{as} \quad \left|h\right|_{L^2(0, T; V)} \rightarrow 0.
\end{align}
Then, the function $z$ will serve as the Fréchet derivative of $S$ at $\bar{U} $ in the direction of $h$, denoted as $z = S'(\bar{U} )h$.\\
Consider $\left|y-\bar{y}-z\right|_X=\left|r\right|_X$. To estimate this norm we first take 
\begin{align*}
     r_{t}+A r+B^{\prime}(\bar{y}) r-|\nabla \bar{y}|^{2} r -2\langle\nabla \bar{y} , \nabla r\rangle \bar{y}= -B(d)+|\nabla d|^{2} d 
 +2\langle\nabla d, \nabla \bar{y}\rangle d+|\nabla d|^{2} \bar{y}
\end{align*}
Let us take the inner product with $Ar$ and then
$$
\begin{aligned}
 \left\langle r_{1}, A r\right\rangle=-|A r|^{2}-\left\langle B^{\prime}(\bar{y}) r, A r\right\rangle+&|\nabla \bar{y}|^{2}\langle r, A r\rangle  +2\langle\nabla \bar{y}, \nabla r\rangle\langle\bar{y}, A r\rangle  +|\nabla d|^{2}\langle d, A r\rangle\\&-\langle B(d), A r\rangle +2\langle\nabla d, \nabla \bar{y}\rangle\langle d, A r\rangle  +|\nabla d|^{2}\langle\bar{y}, A r\rangle.
\end{aligned}
$$
Since $B^{\prime}(\bar{y}) r=B(\bar{y}, r)+B(r, \bar{y})$. So $\left\langle B^{\prime}(\bar{y}) r, A r\right\rangle=b(\bar{y}, r, A r)+b(r, \bar{y}, A r\rangle$

$$
\text { and }\left|\left\langle B^{\prime}(\bar{y}) r, A r\right\rangle\right| \leq \sqrt{2}|\bar{y}|_{H}^\frac{1} {2}|\bar{y}|_{V}^\frac{1} {2}|r|_{V}^\frac{1} {2}|r|_{E}^\frac{1} {2}|A r|+\sqrt{2}|r|_{H}^\frac{1} {2}|r|_{V}^\frac{1} {2}|\bar{y}|_{V}^\frac{1} {2}|\bar{y}|_{E}^\frac{1} {2}|A r|.
$$
Since $\bar{y}=S(\bar{U} )$ and therefore $|\bar{y}|_{H}^{2} \leq 1$. By similar argument  $|y|_{H}^{2} \leq 1$. So $|d|_{H}^{2} \leq 2$.
As $\bar{y} \in L^{\infty}(0, T ; V) \cap L^{2}(0, T; D(A))$
$$
\Rightarrow|\bar{y}|_{D(A)}<\infty.
$$
Hence
$$
\begin{aligned}
& \left|\left\langle B^{\prime}(\bar{y}) r, A r\right\rangle\right| \leq C|A r|^{2}.
\end{aligned}
$$
Again
$|\left.|\nabla \bar{y}|^{2}\langle r, A r\rangle|\leq C| A r\right|^{2}$ (By Cauchy Schwartz inequality).\\
Moreover,
$$
\begin{aligned}
& |2\langle\nabla \bar{y}, \nabla r\rangle\langle\bar{y}, A r\rangle| \leq C|A r|^{2} \;\text{and} \; \left.|| \nabla d\right|^{2}\langle d, r\rangle|\leq C| \nabla d|^{2}|r| \leq C {\varepsilon}|d|_E^{4}+\frac{C}{4 \varepsilon}|A r|^{2}.
\end{aligned}
$$
We have used the Youngs inequality for a given $\varepsilon$. Again by similar arguments, we have,
$$
\begin{aligned}
& |\langle B(d), A r\rangle| \leq C|d|_{E}^{2}|A r| \leq C \varepsilon|d|_{E}^{4}+\frac{C}{4 \varepsilon}|A r|^{2}. \\
& |2\langle\nabla d, \nabla \bar{y}\rangle\langle d, A r\rangle|\leq C{\varepsilon} |d|_{E}^{4}+\frac{C}{4 \varepsilon}|A r|^{2}. \\
& \left.|| \nabla d\right|^{2}\langle \bar{y}, A r\rangle|\leq C \varepsilon| d|^{4}_E+\frac{C}{4 \varepsilon}|A r|^{2}.
\end{aligned}
$$
We select $\varepsilon$ in a manner that ensures the coefficients of $|Ar|^2$ remain negative on the right-hand side.
So, $\frac{1}{2} \frac{d}{d t}|r|_{V}^{2} \leq C|d|_{E}^{4}$ and    $|A r|^{2} \leq C |d|_{E}^{4}$. Performing the integration from $0$ to $T$ yields the following result:
\begin{align*}
& |r|_{X}^{2} \leq C|d|_{X}^{4} \\
& \text { or }|r|_{X} \leq C|d|_{X}^{2}.
\end{align*}
\par By Lipschitz continuity of the solution mapping we get  $|d|_X^2=|y-\bar{y}|_X^2=|S(\bar{U} +h)-S(\bar{U} )|_X^2\leq |h|_{L^2(0, T; V)}^2$. Thus ({\ref{eq:7}}) fulfilled and so $S$ is  Fréchet differentiable and $S'(\bar{U} )h=z$.
\end{proof}
To establish the first-order optimality conditions, it is necessary to have the adjoint operator of $S'(u)$, which is represented as $S'(u)^*$. The investigation of this adjoint mapping was conducted by Hinze \cite{Hinze2002OptimalAI} and Hinze and Kunisch \cite{Hinze2001SecondOM}. The study on this adjoint map has also been carried out and documented in \cite{phdthesis}.
\begin{lemma}\label{lem:1}
   Let $\bar U\in L^2(0, T; V)$. Then $S'(\bar U)^*$ is a continuous linear map from $X^*$ to $L^2(0, T; V)$ . Then for $g\in X^*$, $\lambda=S'(\bar U)^*g$ iff 
   \begin{align*}
       (w_t + Aw+B'(\bar y)w-|\nabla (\bar y)|^2w-2\langle \nabla w,\nabla \bar{y}\rangle, \lambda)_{L^2(0, T;V'),L^2(0, T; V)}=(g,w)_{X^*, X}.
   \end{align*}
   $\forall w\in X$.
\end{lemma}
\begin{proof}
    Consider the linearized equation 
    \begin{align*}
          y_{t}+A y+B^{\prime}(\bar{y} ) y  -|\nabla \bar{y} |^{2}y -2\langle\nabla \bar{y} , \nabla y\rangle\bar{y} =h.
    \end{align*}
    Here $\bar y=S(\bar U)$.
    Let us define the operator $T:X\rightarrow L^2(0, T; V')$ by
    \begin{align*}
        Ty=y_{t}+A y+B^{\prime}(\bar{y} ) y  -|\nabla \bar{y} |^{2}y -2\langle\nabla \bar{y} , \nabla y\rangle\bar{y}.
    \end{align*}
    Hence, the linearized equation can be expressed in the following manner:
    \begin{align*}
        Ty=h.
    \end{align*}
    $T$ is clearly a linear map and $T^{-1}=S'(\bar U)$, so $T^{-1}$ is linear and continuous.
    \par The map $T^*$ is a linear map from $L^2(0, T; V)$ to $X^*$ and its action defined by 
    \begin{align*}
        (T^* v, y)_{X^*,X}= (y_{t}+A y+B^{\prime}(\bar{y} ) y  -|\nabla \bar{y} |^{2}y -2\langle\nabla \bar{y} , \nabla y\rangle\bar{y}, v)_{L^2(0, T;V'),L^2(0, T; V)}
    \end{align*}
    for $v\in L^2(0, T; V)$.\\
    $(T^{-1})^*$ is a linear map from $X^*$ to $L^2(0, T; V)$ and $(T^{-1})^*=S'(\bar U)^*$. Then for $g\in X^*$ there exists $\lambda\in L^2(0, T; V)$ such that 
    $$(T^{-1})^* g=\lambda=S'(\bar U)^*g,\;\text{or } g=T^* \lambda.$$
    since $(T^{-1})^*=(T^*)^{-1}$. Then,
     \begin{align*}\label{ew:1}
        (T^* \lambda, w)_{X^*,X}&=  (w_t + Aw+B'(\bar y)w-|\nabla (\bar y)|^2w-2\langle \nabla w,\nabla \bar{y}\rangle, \lambda)_{L^2(0, T;V'),L^2(0, T; V)}\\&=(g,w)_{X^*, X}.
    \end{align*}
\end{proof}
\section{The optimal control problem}
For the purpose of proving the existence of optimal controls, we can take the cost functional of the form,

\[
J(y,U):=\frac{1}{2} \int_{0}^{T}|A^{1 / 2} y(t)|_{H}^{2} dt+\frac{1}{2} \int_{0}^{T}|U(t)|_{V}^{2} dt.
\]
We define the set of admissible controls $U_{ad}$ by 
 $$U_{ad}:=\{ U \in T_u \mathrm{M} : |U|_{V} \text{ is bounded}\}$$.
 \par The optimization problem is
 $$\min J(y,U)$$
 subject to the state equation 
 \begin{align*}
      y_{t}+A y+B(y)-|\nabla y|^{2} y={U}
      y(0)=y_0
 \end{align*}
 $U\in U_{ad}$
 \subsubsection{Existence of solutions}
\begin{theorem}{\label{thm:10}}
The optimal control problem admits a globally optimal solution $U \in U_{ad}$ with an associated state $y \in L^2(0, T; E) \cap L^{\infty}(0, T; V)$.
\end{theorem}

\noindent Proof: Let $y$ be the solution of the following system,

\[
\begin{aligned}
& y_{t}+A y+B(y)-|\nabla y|^{2} y=U , \\
& y(0)=u_{0}.
\end{aligned}
\]

The space $U_{ad}:=\{ U \in T_u \mathrm{M}: |U|_{V} \text{ is bounded}\}$.

First, we note that for each $U \in L^{2}(0, T ; V)$, we get a unique solution $y \in L^{\infty}(0, T ; V) \cap L^{2}(0, T ; D(A))$ such that $J(U) < \infty$.

For each such admissible pair,

\[
\begin{aligned}
&M_{t}(y, U, v)=0 \quad \forall \;v \in C_{c}^{\infty}[0, T] \text{ where all the derivatives of $v$ are bounded}.
\end{aligned}
\]

\[
\begin{aligned}
\text{Where } M_{t}(y, U, v)=\langle y(t), v\rangle
+&\int_{0}^{t}\left\langle A y(r)+B(y(r))_{-}\right. \left.|\nabla y(r)|^{2} y(r)-U, v\right\rangle dt\\&-\left\langle y_{0}, v\right\rangle.
\end{aligned}
\]

Clearly, $0 \leq J(U)$ for each admissible pair $(y, U)$. Hence, there exists an infimum of $J$ over all admissible controls and states,

\[
0 \leq \bar{J}:=\inf_{U \in U_{ad}} J(U) < \infty.
\]

Moreover, there is a sequence $\left(y_{n}, U_{n}\right)$ of admissible pairs such that $J\left(y_{n}, U_{n}\right) \longrightarrow \bar{J}$ as $n \rightarrow \infty$. The set $\{U_n\}$ is bounded in $U_{ad}$, so ${y_n}$ is bounded in $L^{\infty}(0, T ; V) \cap L^{2}(0, T ; D(A))$. Therefore, we can extract a subsequence $(y'_n, U'_n)$ converging weakly to some limit $(y, U)$. Since the space $U_{ad}$ is closed and convex, $U \in U_{ad}$. We have term-by-term convergence, so $M_{t}(y, U, v) = 0$. Hence, $(y, U)$ is admissible.

Note that the functional 

\[
F(y, U):=\frac{1}{2} \int_{0}^{T}|A^{1 / 2} y(t)|_{H}^{2} dt+\frac{1}{2} \int_{0}^{T}|U(t)|_{V}^{2} dt
\]

is convex, continuous, and hence weakly sequentially lower semicontinuous. So we have $F(y, U) \leq \lim_{n\rightarrow\infty} \inf F(y_n,U_n)$. Thus we have 

\[
J(y,U) \leq \bar{J}.
\]

Since $(y, U)$ is admissible and $\bar{J}$ is the infimum over all admissible pairs, it follows that $\bar{J} = J(y, U)$. Hence the claim is proved.

\subsection{Lagrange functional} 
We aim to define the Lagrange functional $\mathcal{L}: X\times{L^2(0, T; V)}\times{L^2(0, T; V)}$ for the optimal control problem as follows:
\begin{align}
    \mathcal{L}(y, U, \lambda)=J(y,U)-(y_t + Ay+ B(y)-|\nabla y|^2y-U, \lambda)_{L^2(0, T;V'),L^2(0, T; V)}
\end{align}
The first-order derivative of $\mathcal{L}$ with respect to $y $ and $U$ in the direction of $w\in X$ and $h\in L^2(0, T; V)$ are denoted by $\mathcal{L}_y(y,U,\lambda)w$ and $\mathcal{L}_U(y,U,\lambda)h$ respectively and 
\begin{align*}
 \mathcal{L}_y(y,U,\lambda)w&=-(w_t + Aw+B'( y)w-|\nabla (y)|^2w-2\langle \nabla w,\nabla {y}\rangle, \lambda)_{L^2(0, T;V'),L^2(0, T; V)} \\&+ \langle y, w\rangle_{L^2(0, T; V)} ,\\
 \mathcal{L}_U(y,U,\lambda)h&= \langle U, h\rangle_{L^2(0, T; V)} + (U,\lambda)_{L^2(0, T;V),L^2(0, T; V')}.
\end{align*}
\subsection{First order necessary optimality conditions}
First-order necessary optimality conditions can be found in many literature sources. One can follow the \cite{articlee} and \cite{Trltzsch2006SecondorderSO} for more details. The necessary optimality conditions can be obtained by applying the formal Lagrange method. For more detailed information on the formal Lagrange method, refer to section 2.10 of \cite{tröltzsch2010optimal}. Now, we will state and demonstrate the first-order optimality condition.
 \begin{theorem}{(\textbf{Necessary condition})}.
     Let $\bar U$ be locally optimal in $L^2(0, T; V)$ with associated state $\bar y=S(\bar U).$ Then there exists $\lambda\in L^2(0, T; V)$ such that
     \begin{align*}
         \mathcal{L}_y(\bar y,\bar U,\lambda)w&=0\quad\quad\forall w\in X,\\
         \mathcal{L}_U(\bar y,\bar U,\lambda)(U-\bar U)&\ge 0\quad\quad\forall U\in U_{ad}.
     \end{align*}
 \end{theorem}
 \begin{proof}
     We will consider $\lambda=S'(\bar U)\bar y$. Then 
     \begin{align*}
         \mathcal{L}_y(\bar y,\bar U,\lambda)w&=-(w_t + Aw+B'( \bar y)w-|\nabla (\bar y)|^2w-2\langle \nabla w,\nabla {\bar y}\rangle, \lambda)_{L^2(0, T;V'),L^2(0, T; V)} \\&+ \langle \bar y, w\rangle_{L^2(0, T; V)} 
     \end{align*}
   Utilizing the construction of $\lambda$ and the provided lemma \ref{lem:1}, we have \\
   \quad\quad\quad$\mathcal{L}_y(\bar y,\bar U,\lambda)w=0$ for all $w\in X.$
   \par Using Theorem (2.22) of \cite{tröltzsch2010optimal} and using the same construction of $\lambda$ we have, $\mathcal{L}_U(\bar y,\bar U,\lambda)(U-\bar U)=\langle \bar U, U-\bar U\rangle_{L^2(0, T; V)} + (\bar U,\lambda)_{L^2(0, T;V),L^2(0, T; V')}\ge 0$.
 \end{proof}
\bibliographystyle{plain}
\bibliography{main} 
\end{document}